\documentclass[12pt]{amsart}
\usepackage{amscd,setspace,amssymb,amsopn,amsmath,amsthm,mathrsfs,lmodern,graphics,amsfonts,enumerate,verbatim,calc
} 
\usepackage[all]{xy}
\usepackage{fullpage}
\usepackage[colorlinks=true, citecolor=PineGreen, filecolor=PineGreen, linkcolor=Purple,pagebackref, hyperindex]{hyperref}
\usepackage[dvipsnames]{xcolor}
\usepackage{todonotes}
\usepackage{tikz-cd}
\usepackage{color, hyperref}
\usepackage{amssymb,amsmath}
\newtheorem{Theoremx}{Theorem}
 
\newtheorem{theorem}{Theorem}[section]
\newtheorem{lemma}[theorem]{Lemma}

\newtheorem{corollary}[theorem]{Corollary}

\newtheorem{question}{Question}

\theoremstyle{definition}

\theoremstyle{remark}

\newcommand{\Cl}{\operatorname{Cl}}

\newcommand{\Spec}{\operatorname{Spec}}

\newcommand{\Tor}{\operatorname{Tor}}
\newcommand{\Hom}{\operatorname{Hom}}

\newcommand{\e}{\operatorname{e_{HK}}}

\newcommand{\fm}{\mathfrak{m}}

\begin{document}
\title{A Theorem about maximal Cohen-Macaulay modules}

\author[Thomas Polstra]{Thomas Polstra}
\thanks{Polstra was supported in part by NSF Postdoctoral Research Fellowship DMS $\#1703856$.}
\address{Department of Mathematics, University of Utah, Salt Lake City, UT 84102 USA}
\email{polstra@math.utah.edu}

%\thanks{2010 {\em Mathematics Subject Classification\/}:13A35, 13D45.}
%\keywords{}

\begin{abstract} 
It is shown that for any local strongly $F$-regular ring there exists natural number $e_0$ so that if $M$ is any finitely generated maximal Cohen-Macaulay module, then the pushforward of $M$ under the $e_0$th iterate of the Frobenius endomorphism contains a free summand. Consequently, the torsion subgroup of the divisor class group of a local strongly $F$-regular ring is finite.
\end{abstract}

\maketitle

\section{Introduction}\label{Section introduction}

Let $R$ be a Noetherian commutative ring of prime characteristic $p>0$ and for each $e\in\mathbb{N}$ let $F^e\colon R\to R$ be the $e$th iterate of the Frobenius endomorphism of $R$. Throughout this article we assume $R$ is $F$-finite, i.e. the Frobenius endomorphism is a finite map. If $M$ is an $R$-module then $F^e_*M$ is the $R$-module obtained from $M$ via restriction of scalars under $F^e$.  The ring $R$ is said to be strongly $F$-regular if for each nonzero $r\in R$ there exists an $e\in\mathbb{N}$ and $\varphi\in \Hom_R(F^e_*R,R)$ such that $\varphi(F^e_*r)=1$. Every local strongly $F$-regular ring is a Cohen-Macaulay normal domain. In particular, studying the class of Cohen-Macaulay modules in a strongly $F$-regular ring is a warranted venture. The main contribution of this article is the following uniform property concerning the class of finitely generated Cohen-Macaulay modules in strongly $F$-regular rings. 

\begin{Theoremx}
\label{Main result}
Let $(R,\fm,k)$ be a local $F$-finite and strongly $F$-regular ring of prime characteristic $p>0$. There exists an $e_0\in \mathbb{N}$ so that if $M$ is any finitely generated maximal Cohen-Macaulay $R$-module then there exists an onto $R$-linear map $F^{e_0}_*M\to R$, i.e. $R$ can be realized as direct summand of $F^{e_0}_*M$.
\end{Theoremx}

More precisely, we show any minimal generator of a maximal Cohen-Macaulay module $M$ can be sent to $1$ in $R$ under an $R$-linear map $F^{e_0}_*M\to R$ in Theorem~\ref{Cohen-Macaulay Lemma}. A new application of Theorem~\ref{Main result} is that the torsion subgroup of the divisor class group of any local strongly $F$-regular ring is finite, see Corollary~\ref{corollary finite torsion subgroup}. When coupled with results originating from \cite{Carvajal-RojasSchwedeTucker,Carvajal-Rojas}, it is now known that the torsion group of the divisor class group of a strongly $F$-regular ring is finite and every torsion divisor has order bounded by the reciprocal of the $F$-signature of $R$.

We utilize Theorem~\ref{Main result} to rederive two other important  properties of local strongly $F$-regular rings. First and foremost we reprove that the divisor class group of a $2$-dimensional strongly $F$-regular ring is finite in Corollary~\ref{corollary Lipman}. It is important to note that our proof of this property does not require an understanding of resolution of singularities of excellent surfaces with at worst rational singularities by quadratic transforms. We also show that the $F$-signature of a local strongly $F$-regular ring is positive, recapturing the main result of \cite{AberbachLeuschke} in an elementary, novel, and streamlined fashion.

\section{Preliminary results and notation}\label{Section prelims}

 Suppose $(R,\fm,k)$ is a local $F$-finite strongly $F$-regular ring. Every strongly $F$-regular ring is a normal domain and therefore has a well-defined divisor class group on $X=\Spec(R)$. For the sake of convenience, we recall several elementary properties concerning divisors in normal rings of prime characteristic. For each Weil divisor $D$ on $X$ we write $R(D)$ for the global sections of $\mathcal{O}_X(D)$. We refer to $R(D)$ as a divisorial ideal. Every divisorial ideal is a rank $1$ module satisfying Serre's condition $(S_2)$. Conversely, every rank $1$ module satisfying Serre's condition $(S_2)$ is isomorphic to a divisorial ideal.  Suppose $D_1, D_2$ are Weil divisors on $X$. Since $X$ is affine, the divisors $D_1, D_2$ are linearly equivalent if and only if $R(D_1)\cong R(D_2)$. The module $\Hom_R(R(D_1),R(D_2))$ is isomorphic to the divisorial ideal $R(D_2-D_1)$. Moreover, the reflexification\footnote{The reflexification of an $R$-module $M$ is the module $\Hom_R(\Hom_R(M,R),R)$. If $R$ is normal and $M$ is a finitely generated $(S_2)$ module then the natural map $M\to \Hom_R(\Hom_R(M,R),R)$ is an isomorphism, \cite[Theorem~1.9]{Hartshorne}.} of $R(D_1)\otimes_R R(D_2)$ is isomorphic to $R(D_1+D_2)$. In particular, since reflexification commutes with restriction of scalars under Frobenius, the reflexification of $(F^{e}_*R(D_1))\otimes_R R(D_2)$ is isomorphic to $F^{e}_*R(D_1+p^{e}D_2)$.

 Now suppose that $(R,\fm,k)$ is a Cohen-Macaulay normal $F$-finite domain, e.g. strongly $F$-regular, of prime characteristic $p>0$. Every $F$-finite ring is the homomorphic image of a regular ring by \cite[Rem.~13.6]{Gabber} and therefore has a canonical module $\omega_R$. Because the canonical module of a normal domain is $(S_2)$ and rank $1$ we have that $\omega_R\cong R(K_X)$ for some Weil divisor $K_X$. We refer to such a divisor as a canonical divisor. 
 
 %We remind the reader that if $M$ is a finitely generated maximal Cohen-Macaulay module then $\Ext^i_R(M,R(K_X))=0$ for all $i>0$ and $\Hom_R(M,R(K_X))$ is maximal Cohen-Macaulay, see \cite[Theorem~3.3.10]{BrunsHerzog}. In particular, Hom-ing a short exact sequence of finitely generated maximal Cohen-Macaulay modules into $R(K_X)$ produces a dual short exact sequence of finitely generated maximal Cohen-Macaulay modules.

 The following lemma is characteristic-free and is an observation that $(S_2)$-modules satisfy a Krull-Schmidt type condition for rank $1$-summands.

\begin{lemma}
\label{Lemma Krull-Schmidt for rank 1 summands}
Let $(R,\fm,k)$ be a normal local domain. Let $C$ be a finitely generated $(S_2)$-module, $M$ a rank $1$ module, and suppose that $C\cong M^{\oplus a_1}\oplus N_1\cong M^{\oplus a_2}\oplus N_2$ are choices of direct sum decompositions of $C$ so that $M$ cannot be realized as a direct summand of either $N_1$ or $N_2$. Then $a_1=a_2$.
\end{lemma}

\begin{proof}
Without loss of generality we may assume $a_1>0$ or $a_2>0$. Because $C$ satisfies Serre's condition $(S_2)$, the rank $1$ module $M$ is $(S_2)$, and therefore $M\cong R(D)$ for some Weil divisor $D$. The divisorial ideal $R(D)$ is a direct summand of a finitely generated $(S_2)$ module $N$ if and only if $R$ is a direct summand $\Hom_R(N,R(D))$. Indeed, this follows from the fact that the natural map $N\to \Hom_R(\Hom_R(N,R(D)),R(D))$ is an isomorphism on the regular locus and therefore an isomorphism globally by \cite[Proposition~1.11]{Hartshorne}. It follows that $\Hom_R(C,R(D))\cong R^{\oplus a_1}\oplus \Hom_R(N_1,R(D))\cong R^{\oplus a_2}\oplus \Hom_R(N_2,R(D))$ are choices of direct sum decompositions of $\Hom_R(C,R(D))$ where neither $\Hom_R(N_1,R(D))$ nor $\Hom_R(N_2,R(D))$ have a free summand. In other words, $a_1$ and $a_2$ are the maximal number of free summands appearing in choices of direct sum decompositions of $\Hom_R(C,R(D))$. But such numbers can be computed after completion and hence $a_1=a_2$ since complete local rings satisfy the Krull-Schmidt condition. \end{proof}

\begin{corollary}
\label{Corollary Krull-Schmidt for rank 1 summands}
Let $(R,\fm,k)$ be a local normal domain and $C$ a finitely generated $(S_2)$-module. If $D_1,D_2,\ldots, D_t$ are divisors representing distinct elements of the divisor class group of $R$ and $R(D_i)$ is a direct summand of $C$ for each $1\leq i \leq t$, then $R(D_1)\oplus \cdots \oplus R(D_t)$ is a direct summand of $C$.
\end{corollary}

\begin{proof}
Let $1\leq i\leq t-1$ and suppose we have found direct sum decomposition 
\[
C\cong R(D_1)\oplus \cdots \oplus R(D_i)\oplus N.
\]
We claim that $R(D_{i+1})$ is a direct summand of $N$. By Lemma~\ref{Lemma Krull-Schmidt for rank 1 summands} it is enough to show that $R(D_{i+1})$ is not a summand of $R(D_1)\oplus \cdots \oplus R(D_i)$. By Hom-ing into $R(D_{i+1})$ it is enough to show $R$ is not a direct summand of $R(D_1-D_{i+1})\oplus \cdots \oplus R(D_i-D_{i+1})$. Suppose by way of contradiction that $R(D_1-D_{i+1})\oplus \cdots \oplus R(D_i-D_{i+1})$ had $R$ as a summand. Then there exists an onto $R$-linear map $R(D_1-D_{i+1})\oplus \cdots \oplus R(D_i-D_{i+1})\to R$. Since $R$ is local there exists $1\leq j\leq i$ so that the image of $R(D_j-D_{i+1})$ in $R$ contains a unit. Hence $R$ is a direct summand of $R(D_j-D_{i+1})$. By rank considerations $R(D_j-D_{i+1})\cong R$, i.e. $D_j$ and $D_{i+1}$ are linearly equivalent divisors, a contradiction to our assumption that $D_j$ and $D_{i+1}$ represent distinct elements of the divisor class group.
\end{proof}

We return to the assumption that $(R,\fm,k)$ is a local ring of prime characteristic $p>0$. If $M$ is a finitely generated $R$-module and $e\in \mathbb{N}$ then we let 
\[
I_e(M)=\{\eta\in M\mid \varphi(F^e_*\eta)\in \fm , \forall \varphi\in\Hom_R(F^e_*M,R)\}.
\]
Observe that an element $\eta\in M$ avoids $I_e(M)$ if and only if there exists $\varphi\in \Hom_R(F^e_*M,R)$ so that $\varphi(F^e_*\eta)=1$. If $M=R$ then we let $I_e=I_e(R)$. The ring $R$ is strongly $F$-regular if and only if $\bigcap_{e\in\mathbb{N}}I_e=0$. We record some basic properties of these subsets of $M$.

\begin{lemma}
\label{splitting ideals} Let $(R,\fm,k)$ be a local $F$-finite ring of prime characteristic $p>0$ and $M$ a finitely generated $R$-module.
\begin{enumerate}
\item The sets $I_e(M)$ form submodules of $M$.
\item For each $e\in \mathbb{N}$ there is an inclusion $\fm^{[p^e]}M\subseteq I_e(M)$.
\item If $e\geq e'$ then $I_{e}(M)\subseteq I_{e'}(M)$.
\item If $R$ is strongly $F$-regular and $M$ is torsion-free then $\bigcap_{e\in \mathbb{N}} I_e(M)=0$.
\end{enumerate}
\end{lemma}

\begin{proof}
Properties (1)-(3) are straight-forward and left to the reader to verify. For $(4)$ let $\eta\in M$ be a nonzero element and consider an $R$-linear map $\psi\colon M\to R$ such that $\psi(\eta)=r\not = 0$. We are assuming $R$ is strongly $F$-regular, i.e. $\bigcap_{e\in\mathbb{N}}I_e=0$. Hence $r\not \in I_e$ for some $e$ and therefore there exists $\varphi\in \Hom_R(F^e_*R,R)$ so that $\varphi(F^e_*r)=1$. In particular, $\varphi\circ F^e_*\psi\in \Hom_R(F^e_*M,R)$ and $\varphi\circ F^e_*\psi(\eta)=1$. 
\end{proof}

\section{Main results}\label{Section results}

\begin{theorem}
\label{Cohen-Macaulay Lemma}
Let $(R,\fm,k)$ be an $F$-finite and strongly $F$-regular ring of prime characteristic $p>0$. There exists an $e_0\in \mathbb{N}$, depending only on $R$, so that if $M$ is a finitely generated maximal Cohen-Macaulay $R$-module and $\eta\in M \setminus \fm M$, then there exists $\varphi\in \Hom_R(F^{e_0}_*M,R)$ such that $\varphi(F^{e_0}_*\eta)=1$.
\end{theorem}

\begin{proof}
Begin by surjecting a finitely generated free module $R^{\oplus n}$ onto $\Hom_R(M,R(K_X))$ and then consider a short exact of the following form:
\[
0\to N\to R^{\oplus n}\to \Hom_R(M,R(K_X))\to 0.
\]
The modules $\Hom_R(M,R(K_X))$ and $N$ are maximal Cohen-Macaulay. Hence there is a short exact sequence 
\[
0\to M\to R(K_X)^{\oplus n}\to \Hom_R(N,R(K_X))\to 0
\] 
obtained by applying $\Hom_R(-,R(K_X))$, see \cite[Theorem~3.3.10]{BrunsHerzog}. Let $C$ denote the maximal Cohen-Macaulay $R$-module $\Hom_R(N,R(K_X))$ and let $\underline{x}=x_1,\ldots,x_d$ be a full system of parameters of $R$. Then $\Tor_1^R(R/(\underline{x}), C))=0$ since $C$ is maximal Cohen-Macaulay and $\Tor_1^R(R/(\underline{x}), C))$ agrees with the first Koszul homology on the $C$-regular sequence $x_1,\ldots, x_d$. Hence there is short exact sequence
\[
0\to \frac{M}{(\underline{x})M}\to \frac{R(K_X)^{\oplus n}}{(\underline{x})R(K_X)^{\oplus n}}\to \frac{C}{(\underline{x})C}\to 0.
\]
In particular, under the inclusion $M\to R(K_X)^{\oplus n}$ we find that 
\begin{align}\label{Trivial Artin-Rees}
(\underline{x})M=M\cap (\underline{x}) R(K_X)^{\oplus n}.
\end{align}

By (3) and (4) of Lemma~\ref{splitting ideals}, utilizing faithfully flat descent to the completion of $R$, and Chevalley's Lemma, \cite[Section~2, Lemma~7]{Chevalley}, there exists an $e_0\in \mathbb{N}$ so that $I_{e_0}(R(K_X))\subseteq (\underline{x})R(K_X)$. Suppose that $\eta \in M\setminus \fm M$. Then under the inclusion $M\to R(K_X)^{\oplus n}$ we find that $\eta\not \in (\underline{x})R(K_X)^{\oplus n}$ by (\ref{Trivial Artin-Rees}). In particular, $\eta$ avoids $I_{e_0}(R(K_X))^{\oplus n}$ and hence there is a commutative diagram 
\[
\begin{tikzcd}
F^{e_0}_*M \arrow[r] \arrow[dr, "\varphi"']
& F^{e_0}_*R(K_X) \arrow[d, "\psi"]\\
& R
\end{tikzcd}
\]
so that $\varphi(F^{e_0}_*\eta)=\psi(F^{e_0}_*\eta)=1$.
\end{proof}

The first application given of Theorem~\ref{Cohen-Macaulay Lemma} is that the divisor class group of a $2$-dimensional strongly $F$-regular ring is finite. Before this article, the proof that $2$-dimensional strongly $F$-regular rings would have gone as follows: every strongly $F$-regular ring is weakly $F$-regular, in particular is $F$-rational. Hence $R$ has pseudorational singularities by \cite[Theorem~3.1]{Smith}. But every $F$-finite ring is excellent by \cite[Theorem~1]{KunzExcellent} and therefore has rational singularities since resolution of singularities of excellent surfaces is known by \cite{LipmanDesingularization}. Every $2$-dimensional rational surface singularity has finite divisor class group by \cite[Proposition~17.1]{LipmanRational}, a result that requires classifying minimal resolution of singularities of rational surfaces by quadratic transforms.

For the following corollaries we remind the reader that divisorial ideals in a normal domain $R$ are torsion-free and therefore have the same dimension as $R$.

\begin{corollary}
\label{corollary Lipman}
Let $(R,\fm,k)$ be an $F$-finite and strongly $F$-regular ring of prime characteristic $p>0$ and Krull dimension $2$. Then the divisor class group of $R$ is finite.
\end{corollary}

\begin{proof}
Every divisorial ideal in a $2$-dimensional normal domain is Cohen-Macaulay. By Theorem~\ref{Cohen-Macaulay Lemma} there is an $e_0\in \mathbb{N}$ so that if $D$ is any Weil divisor then $F^{e_0}_*R(-p^{e_0}D)$ has a free summand. Reflexifying $F^{e_0}_*R(-p^{e_0}D)\otimes_R R(D)$ then shows $R(D)$ is a direct summand of $F^{e_0}_*R$. Therefore every divisorial ideal can be realized as direct summand of $F^{e_0}_*R$. By Corollary~\ref{Corollary Krull-Schmidt for rank 1 summands}, if $D_1,\ldots, D_t$ are divisors, no two of which are linearly equivalent, then $R(D_1)\oplus\cdots \oplus R(D_t)$ is a direct summand of $F^{e_0}_*R$. By rank considerations on $F^{e_0}_*R$ there can only be finitely many many Weil divisors up to linear equivalence. 
\end{proof}

\begin{corollary}
\label{corollary finite torsion subgroup}
Let $(R,\fm,k)$ be an $F$-finite and strongly $F$-regular ring of prime characteristic $p>0$. Then the torsion subgroup of the divisor class group of $R$ is finite.
\end{corollary}

\begin{proof}
If $D$ is torsion divisor in a strongly $F$-regular ring $R$ then the fractional ideal $R(D)$ is Cohen-Macaulay, see \cite[Corollary~3.3]{PatakfalviSchwede} and \cite[Corollary~3.12]{DaoSe}, cf \cite[Corollary~2.9]{Watanabe}. By Theorem~\ref{Cohen-Macaulay Lemma} there exists an $e_0\in \mathbb{N}$ so that if $D$ is a torsion divisor then $F^{e_0}_*R(-p^{e_0}D)$ has a free summand. The proof now follows as in the proof of Corollary~\ref{corollary Lipman}. 
\end{proof}

Finiteness of the prime-to-$p$ subgroup of the torsion group of an $F$-finite local ring with algebraically closed residue field is something observed in \cite[Corollary~5.1]{Carvajal-Rojas}. However, the techniques of Corollary~\ref{corollary finite torsion subgroup} are more elementary, do not require the residue field to be algebraically closed, and show finiteness of the entire torsion subgroup.

The last result we recapture is Aberbach's and Leuschke's theorem that the $F$-signature of a local strongly $F$-regular ring is positive. The $F$-signature of a ring of prime characteristic is the asymptotic measurement of the number of free summands of $F^e_*R$ as compared to the generic rank of $F^e_*R$. The study of $F$-signature originates in \cite{SmithVanDenBergh, HunekeLeuschke}.  When $(R,\fm,k)$ is local of dimension $d$ and $\ell(M)$ denotes the length of finite length $R$-module $M$, then the $F$-signature of $R$ is realized as the limit
\[
s(R)=\lim_{e\to \infty}\frac{\ell(R/I_e)}{p^{ed}},
\]
a limit which always exists by \cite{Tucker}, cf \cite{PolstraTucker}. 

Before reestablishing the positivity of the $F$-signature of strongly $F$-regular rings we briefly discuss the related prime characteristic invariant Hilbert-Kunz multiplicity. If $(R,\fm,k)$ is local ring of prime characteristic $p>0$ and Krull dimension $d$ then the Hilbert-Kunz multiplicity of $R$ is the limit
\[
\e(R)=\lim_{e\to \infty}\frac{\ell(R/\fm^{[p^e]})}{p^{ed}},
\]
a limit which always exists by \cite{Monsky}. The Hilbert-Kunz multiplicity of a local ring is bounded below by $1$. With this simple fact in mind we are prepared to recapture \cite[Main Result]{AberbachLeuschke}.

\begin{corollary}
\label{corollary Aberbach and Leuschke}
Let $(R,\fm,k)$ be an $F$-finite and strongly $F$-regular ring of prime characteristic $p>0$ and Krull dimension $d$. Then $s(R)>0$.
\end{corollary}

\begin{proof}
Let $e_0$ be as in Theorem~\ref{Cohen-Macaulay Lemma}. Then $I_{e+e_0}\subseteq \fm^{[p^e]}$ for all $e\in \mathbb{N}$. Indeed, if $r\in R\setminus \fm^{[p^e]}$ then $F^e_*r\in F^e_*R\setminus \fm F^{e}_*R$. Hence by Theorem~\ref{Cohen-Macaulay Lemma} there exists an $R$-linear map $F^{e+e_0}_*R\to R $ mapping $F^{e+e_0}_*r$ to $1$, i.e., $r\not \in I_{e+e_0}$. 

It now follows that the $F$-signature of $R$ is positive since
\[
s(R)=\lim_{e\to \infty}\frac{\ell(R/I_{e+e_0})}{p^{(e+e_0)d}}\geq \lim_{e\to \infty}\frac{\ell(R/\fm^{[p^{e}]})}{p^{(e+e_0)d}}=\frac{\e(R)}{p^{e_0d}}>0.
\]
\end{proof}

 Let $(R,\fm,k)$ be a local normal domain with divisor class group $\Cl(R)$. If $R$ is essentially of finite type over $\mathbb{C}$ and has rational singularities then $\Cl(R)\otimes_\mathbb{Z}\mathbb{Q}$ is finitely generated over $\mathbb{Q}$, see \cite[Lemma~1]{Kawamata}. Since $F$-rational rings have pseudorational singularities one might hope an analogous statement holds for local $F$-finite rings with $F$-rational singularities. In particular, one might also expect local strongly $F$-regular rings to have finitely generated divisor class groups by Corollary~\ref{corollary finite torsion subgroup}.

\begin{question}
\label{Question finite generation of divisor class group}
Let $(R,\fm,k)$ be a local $F$-finite and $F$-rational ring of prime characteristic $p>0$ with divisor class group $\Cl(R)$. Is $\Cl(R)\otimes_\mathbb{Z}\mathbb{Q}$ a finitely generated $\mathbb{Q}$-vector space? Moreover, do local strongly $F$-regular rings have finitely generated divisor class group?
\end{question}

 Utilizing results of Boutot, \cite{Boutot}, along with local uniformization of threefolds, \cite{Abhyankar,CossartPiltant},  Javier Carvajal-Rojas and Axel St\"{a}bler have recently answered Question~\ref{Question finite generation of divisor class group} for $3$-dimensional $F$-finite rings of prime characteristic with perfect residue field, see \cite[Section~4]{Carvajal-RojasStabler}. However, even if local uniformization is shown to hold in higher dimensions, the methods of \cite{Carvajal-RojasStabler} will not suffice to answer Question~\ref{Question finite generation of divisor class group} for rings of dimension greater than $3$.

\bibliographystyle{alpha}
\bibliography{References}

\end{document}